\documentclass[journal,twoside,web]{article}

    \usepackage{amsmath,amssymb,amsfonts
    }  
    \usepackage{graphicx,wrapfig}
 
  \usepackage{xcolor}
  \usepackage{tikz}
    \usetikzlibrary{arrows}

   \usepackage{enumitem}

\makeatletter
\newcommand*\rel@kern[1]{\kern#1\dimexpr\macc@kerna}
\newcommand*\widebar[1]{%
  \begingroup
  \def\mathaccent##1##2{%
    \rel@kern{0.8}%
    \overline{\rel@kern{-0.8}\macc@nucleus\rel@kern{0.2}}%
    \rel@kern{-0.2}%
  }%
  \macc@depth\@ne
  \let\math@bgroup\@empty \let\math@egroup\macc@set@skewchar
  \mathsurround\z@ \frozen@everymath{\mathgroup\macc@group\relax}%
  \macc@set@skewchar\relax
  \let\mathaccentV\macc@nested@a
  \macc@nested@a\relax111{#1}%
  \endgroup
}
\makeatother
  
  \allowdisplaybreaks
    \newtheorem{remark}{Remark}
    \newtheorem{ass}{Assumption}
    
    \newtheorem{lem}{Lemma}
    \newenvironment{lemma}{\begin{lem}}{\hfill {\small $\square$}\end{lem}}
    \newtheorem{clm}{Claim}
    
    \newtheorem{fct}{Fact}
    \newenvironment{fact}{\begin{fct}}{\hfill {\small $\square$}\end{fct}}

    \newtheorem{prop}{Proposition}
    
    \newtheorem{corol}{Corollary}
    
    \newtheorem{thm}{Theorem}
    
\newenvironment{proof}[1][\!]{{\it Proof\,#1: }}{{\hfill \mbox{\footnotesize$\blacksquare$}}}

    \definecolor{myred}{rgb}{0.86,0.1,0.16}
    \definecolor{myblue}{rgb}{0,0,.8}

\begin{document}

\title{\bf On the robustness of networks of heterogeneous semi-passive systems interconnected  over directed graphs}

\author{Anes Lazri \quad  Elena Panteley \quad Antonio  Lor{\'i}a 
	\thanks{A. Lazri  is with L2S, CNRS, Univ Paris-Saclay, France 
          (e-mail: anes.lazri@centralesupelec.fr)
          E. Panteley and A. Lor{\'\i}a are with L2S, CNRS, (e-mail: \{elena.panteley,antonio.loria\}@cnrs.fr). 
}}

\maketitle

\begin{abstract}
  In this short note we provide a proof of boundedness of solutions for a network system composed of heterogeneous nonlinear autonomous systems interconnected over a directed graph.  The sole assumptions imposed are that the systems are semi-passive \cite{Porgmosky_semipassivity} and the graph contains a spanning tree. 
\end{abstract}

 \begin{lemma}\label{lem:bndedness}
Consider a network containing $N$ interconnected dynamical systems\footnote{For clarity of exposition, and without loss of generality, we assume that $x_i \in \mathbb R$. However, all the statements remain true if $x_i \in \mathbb R^n$ for any $n>1$.}
\begin{equation}
  \label{963} \dot x_i = f_i(x_i) + u_i , \quad i\leq N, \quad x_i \in \mathbb R,
\end{equation}
each of which defines a semi-passive map $u_i\mapsto x_i$. Let 
\begin{equation}
  \label{967} u_i = -\sum_{j\in \mathcal N_i} a_{ij}(x_i-x_j), \quad a_{i,j} \geq 0,  
\end{equation}
where for each $i\leq N$, $\mathcal N_i$ denotes the set of nodes $\nu_j$ sending information to the node $\nu_i$. Let this network's topology be defined by a directed graph $G$ containing a  spanning tree. Then, the trajectories $t\mapsto x(t)$, where $x := [x_1\, \cdots\, x_N]$,  solutions to \eqref{963}-\eqref{967} for all $i \leq N$, are globally bounded. 
\end{lemma}
\begin{proof} 
The system \eqref{963}-\eqref{967} in compact form, {\it i.e.,} defining $x:= [x_1 \ \cdots \ x_N]^\top$, becomes  
\begin{equation}
  \label{95}  \dot x = F(x) - Lx,
\end{equation}
where  $F(x):= [f_1(x_1) \,\cdots\, f_N(x_N)]^\top$ and 
\[
[L]_{i,j} =  \left\{ 
\begin{array}{ll} \!\!
  -a_{ij}, & i\neq j \\[2pt]
  \!\!\displaystyle\sum^{N}_{\mbox{\scriptsize 
    $\begin{array}{c}
  \ell = 1\\[-1pt] \ell\neq i
    \end{array}$}
  } \!\!\! a_{i\ell}, & i = j, \quad i,\,j \leq N.
\end{array}\right.
\]

By assumption, the graph $G$ contains a spanning tree. If, in addition, it is strongly connected, the results follows along the lines of the proof of \cite[Proposition 2]{DYNCON-TAC16}. If the graph is not strongly connected, the result follows by observing that by reordering the network’s states, the Laplacian  $L$ may be transformed into that of a connected network that consists in a spanning-tree of strongly-connected sub-graphs. Hence, the transformed Laplacian matrix possesses a convenient lower-block-triangular form (see Lemma 2 below). Then, the statement follows after a cascades argument, from the fact that the trajectories of each strongly-connected sub-graph are bounded and remain bounded under the effect of the interconnections (see Lemma \ref{claim2}).

\begin{lemma}\label{claim0}
Let $L\in \mathbb R^{N\times N}$ be the Laplacian matrix associated to a directed connected graph $G$ that contains a directed spanning tree, but is not strongly connected. 
Then, there exists a permutation matrix $T$ and a number $m \in \{2,3, \ldots, N\}$, such that 
\begin{equation}
  \label{975} T^\top L T  = 
  \begin{bmatrix}
    A_{11} & 0 & \ \cdots & 0 \\[4pt]
    -A_{21} & A_{22} & & \\[-17pt]
           & & \rotatebox{10}{$\ddots$}& \vdots\\[-3pt]
    \vdots & & \rotatebox{10}{$\ddots$}& 0 \\[3pt]
    -A_{m1} &  \cdots & -A_{m, m-1} & A_{m, m} 
  \end{bmatrix},
\end{equation}
where for each $i\in \{2,3,\ldots,m\}$, $A_{ii}\in \mathbb R^{n_i\times n_i}$, $A_{ii} = L_{ii} + D_i$, where $L_{ii}\in \mathbb R^{n_i \times n_i}$ corresponds to the Laplacian of a strongly-connected directed graph, $D_i$ 
is a diagonal matrix of  non-negative entries, and $A_{ij}\in \mathbb R^{n_i\times n_j}$ is such that $D_i \boldsymbol 1_{n_i} = \sum_{j=1}^{i-1} A_{ij} \boldsymbol 1_{n_j}$, with $\boldsymbol 1_{n_i} := [1\,\cdots\, 1]^\top \in \mathbb R^{n_i}$, and $D_1 := 0$.
\end{lemma}

\vskip 3pt
Let Lemma \ref{claim0} generate a permutation matrix $T$ and define $z := T^\top x$. Since $T$ is a permutation matrix it is invertible with $T^{-1} := T^\top$. In turn, $t\mapsto x(t)$ of \eqref{95} is globally bounded if and only if so is $t\mapsto z(t)$, solution to 
\begin{equation}
  \label{111}  \dot z = F_z(z) - T^\top L T z,
\end{equation}
where $F_z(z) := T^\top F(Tz)$. 
\begin{remark}
  Since $T$ is a permutation matrix the $i$th element in the vector $F_z(z)$ depends only on $z_i$, {\it i.e.,}  $F_z(z):= [f_1(z_1) \,\cdots f_N(z_N)]^\top$.
\end{remark}

It is only left to show that $t\mapsto z(t)$ is globally bounded. To that end, we use the lower block-triangular structure of $T^\top L T$. Consider the first $n_1$ equations in \eqref{111}, that is, let $\bar z_1 := [z_1 \,\cdots\, z_{n_1}]^\top$, $\bar F_{z_1}(\bar z_1) := [f_1(z_1) \,\cdots  f_{n_1}(z_{n_1})]^\top$. Then, 
\begin{equation}
  \label{119} \dot{\bar z}_1 = \bar F_{z_1}(\bar z_1) - A_{11} z_1, 
\end{equation}
where, after Lemma \ref{claim0}, $A_{11}$ is the Laplacian of a strongly connected graph (since $D_1 :=0$). It follows that the equation \eqref{119} corresponds to the dynamics of a strongly connected network, whose solutions are globally bounded. The latter follows from the proof of \cite[Proposition 2]{DYNCON-TAC16}\footnote{We invoke the proof of \cite[Proposition 2]{DYNCON-TAC16} and not the statement since it is therein inappropriately assumed that the graph is undirected, but the proof of the statement applies to connected-and-balanced graphs, as well as to strongly-connected ones.}. 

Now, the second set of equations in \eqref{111} corresponds to 
\begin{equation}
  \label{125}  \dot{\bar z}_2 = \bar F_{z_2}(\bar z_2) - L_{22}z_2 - D_2 \bar z_2 + A_{21} \bar z_1, \quad \bar z_2 \in \mathbb R^{n_2}
\end{equation}
where $L_{22}$ is the Laplacian of a strongly connected graph and $D_2 \boldsymbol 1_{n_2} = A_{21} \boldsymbol 1_{n_1}$. Note that \eqref{125} corresponds to the dynamics model of a strongly-connected network of semi-passive systems of the form \eqref{963}-\eqref{967}, of dimension $N = n_2$, with an additional  stabilizing term $- D_2 \bar z_2$, and perturbed by an input $v_1 := \bar z_1$. For such systems, we have the following (see the proof below). 
\begin{lemma}\label{claim2}
Consider a group of $N$ semi-passive systems \eqref{963} with input \eqref{967}, interconnected over a strongly connected directed graph with associated Laplacian $L$. Let $B_j \in \mathbb R^{N\times M_j}$, with $j\leq p$,  and $D\in \mathbb R^{N\times N}$, $D:=\mbox{diag}[d_k]$ be matrices whose entries are non-negative, and such that for any $k\leq N$, $d_k := \sum_{j=1}^p \sum_{\ell = 1}^{M_j} [B_j]_{k\ell}$.  Let $\bar x := [x_1 \,\cdots\, x_N]^\top$, $F(\bar x):= [f_1(x_1) \,\cdots\, f_N(x_N)]^\top$ and $v_j \in \mathbb R^{M_j}$ be external bounded inputs. Then, the trajectories of the perturbed networked system 
\begin{equation}
  \label{121}  \dot{\bar x} = F(\bar x) -L\bar x - D\bar x + \sum_{j\leq p}B_j v_j,
\end{equation}
$t\mapsto \bar x(t)$, are globally bounded. 
\end{lemma}

Now, Lemma \ref{claim2} applies to Eq. \eqref{125} with  $p = 1$, $M_1 = n_1$,  $B_1 := A_{21}$, and the input $v_1 := \bar z_1$, which we established to be bounded. In addition, the $k$th element of $D_2$, denoted $d_{2k}$ satisfies $d_{2k} = \sum_{\ell = 1}^{n_1}[A_{21}]_{k\ell}$, where $[A_{21}]_{k\ell}$ denotes the $\ell$th element of the $k$th row of $A_{21}$. It follows, from Lemma \ref{claim2} that the solutions $t\mapsto \bar z_2(t)$ are globally bounded. In turn, for any $i\leq m$, the $i$th set of equations in \eqref{111} reads
\begin{equation}
  \label{129}  \dot{\bar z}_i = \bar F_{z_i}(\bar z_i) - L_{ii}\bar z_i - D_i \bar z_i + \sum_{j=1}^{i-1}   A_{ij} \bar z_j.
\end{equation}
Equation \eqref{129} is of the form \eqref{121}, with $p = i-1$, $v_j:= \bar z_j$, $B_j := A_{ij}$, and $L_{ii}$ corresponds to the Laplacian of a strongly connected network. For each $k\leq n_i$, the $k$th element in the diagonal of $D_i$ satisfies, by definition, $d_{i_k} =  \sum_{j=1}^{i-1} \sum_{\ell = 1}^{n_j} [A_{ij}]_{k\ell}$, where $[A_{ij}]_{k\ell}$ corresponds to the $\ell$th element in the $k$th row of $A_{ij}$. Therefore,
Invoking Lemma \ref{claim2}, with $\bar x := \bar z_i$,  it follows that $t\mapsto \bar z_i(t)$ is globally bounded. The statement of Lemma \ref{lem:bndedness} follows by applying the previous arguments, sequentially, for each $i \in \{3,4,\cdots,\,m\}$.
\end{proof} 
\vskip 3pt
%

\begin{proof}[\,of Lemma \ref{claim0}]
Consider the following.
\begin{fact}\label{fact1}
If a graph $G$, with Laplacian $L_G$, contains a directed spanning tree and is not strongly connected, then there exists a permutation matrix $P_G$ such that $P_G^\top P_G = I$ and 
\begin{equation}
  \label{187} P_G^{\top} L_G P_G = 
  \begin{bmatrix}
    \phantom{-}Q_{G'} & 0 \\
    - R_{\widebar G'} & S_{\widebar G'} 
  \end{bmatrix},
\end{equation}
where $Q_{G'}\in \mathbb R^{n'\times n'}$, with $n' < N$, is the Laplacian matrix of the largest strongly-connected sub-graph $G' \subset G$, containing  $n'$ nodes, including all the root nodes in $G$. The matrix $S_{\widebar G'} $ satisfies  $S_{\widebar G'}  = L_{\widebar G'}  + D_{\widebar G'} $ where $L_{\widebar G'} $ is a Laplacian matrix associated to the graph ${\widebar G'} := G\backslash G'$ and $D_{\widebar G'} $ is the degree matrix, which is diagonal and contains the weights of the links from $G'$ to ${\widebar G'}$.
\end{fact}

The previous fact is true because if $G$ contains only one spanning tree, say $\mathcal T_G$ with root node $\nu_0$, then $\nu_0$ has no incoming link. Therefore,  we can set $G' := (\{\nu_0\}, \emptyset)$. If $\nu_0$ has incoming links, it necessarily forms part of a strongly-connected graph containing at least two nodes including  $\nu_0$ and a bidirectional link, thereby forming a strongly connected set. The same reasoning holds if $G$ has several spanning trees, in which case the respective roots also make part of $G'$. 

Thus, since by assumption, the graph $G$ contains a directed spanning tree $\mathcal T_G$,  let Fact \ref{fact1} generate the largest strongly connected sub-graph of $G$, which containing all the roots of $G$ and $n_1$ nodes in total and we call $G_1 \subset  G$. Then, let $Q_{G_1} \in \mathbb R^{n_1\times n_1}$ denote the Laplacian associated to $G_1$.  Then, for the block $A_{11}$ in \eqref{975} we set $A_{11} := Q_{G_1}$. That is, $A_{11} \in \mathbb R^{n_1\times n_1}$ is the  Laplacian of a strongly connected graph, as desired. Let $\widebar G_1 := G\backslash G_1$  denote the complement of $G_1$. Fact \ref{fact1} also generates the matrices $R_{\widebar G_1}$ and $S_{\widebar G_1}$. That is, 
\begin{equation}
  \label{204.11} P_G^{\top} L_G P_G = 
  \begin{bmatrix}
    \phantom{-}A_{11} & 0 \\
    - R_{\widebar G_1} & S_{\widebar G_1} 
  \end{bmatrix},
\end{equation}

The off-diagonal entries of the matrix $S_{\widebar G_1}\in \mathbb R^{(N-n_1)\times (N-n_1)}$ are non-positive. They represent edges belonging to the graph ${\widebar G_1}$. Indeed, The matrix $S_{\widebar G_1} = L_{\widebar G_1} + D_{\widebar G_1}$, where $L_{\widebar G_1}$ corresponds to the Laplacian associated to the graph ${\widebar G_1}$ and $D_{\widebar G_1}$ is the degree matrix, which is diagonal positive semidefinite and contains the weights of the links from $G_1$ to ${\widebar G_1}$. The entries in the matrix $R_{\widebar G_1}\in \mathbb R^{(N-n_1)\times n_1}$ represent the outgoing links emanating from nodes belonging to $G_1$ towards nodes in the rest of the graph, {\it i.e.,} ${\widebar G_1}$. If $\widebar G_1$ is strongly connected, the matrix in \eqref{204.11} has the desired structure in \eqref{975} and the proof ends. 
%
%

If ${\widebar G_1}$ is not strongly connected, we look for a permutation matrix $P_{\widebar G_1} \in \mathbb R^{(N-n_1)\times (N-n_1)}$ such that $P_{\widebar G_1}^{\top} L_{\widebar G_1} P_{\widebar G_1}$ has a block-triangular form as in \eqref{187}. To that end, we consider two possibilities depending on whether $\widebar G_1$ contains or not a spanning tree. 

\noindent \underline{\it Case 1}: Assume that $\widebar G_1$ contains a spanning tree, or several. Necessarily, the root of at least one of the trees has an incoming link from $G_1$. Then, let Fact \ref{fact1} generate  the largest strongly-connected graph $G_2 \subset \widebar G_1$, containing $n_2$ nodes, including all the roots in $\widebar G_1$. Also after Fact \ref{fact1} there exists a permutation matrix $P_{\widebar G_1}$ such that 
\begin{equation}\label{203}
  P_{\widebar G_1}^{\top} L_{\widebar G_1}   P_{\widebar G_1} 
  = 
  \begin{bmatrix}
    \phantom{-}Q_{G_2} & 0 \\
    -R_{\widebar G_2} & S_{\widebar G_2}
  \end{bmatrix},
\end{equation}
where $Q_{G_2}\in \mathbb R^{n_2\times n_2}$ is the Laplacian matrix associated to $G_2$, $S_{\widebar G_2} := L_{\widebar G_2} + D_{\widebar G_2}$. Also, we define $\widebar G_2 := \widebar G_1\backslash G_2 =   G \backslash \{G_1 \cup G_2\} $, {\it i.e.,} $\widebar G_2$ contains all the nodes in $G$, but which are not contained in $G_1$ nor in $G_2$. 

Then, we apply the permutation blockdiag$\big[ I_{n1} \ P_{\widebar G_1}^\top \big]$ on the matrix on the right-hand side of \eqref{204.11}. We obtain the matrix 
\begin{align}
  \begin{bmatrix}
     A_{11} & 0  \\ -P_{\widebar G_1}^\top R_{\widebar  G_1} & P_{\widebar G_1}^\top S_{\widebar G_1} P_{\widebar G_1}
  \end{bmatrix},
  \label{204}
\end{align}
which has a lower-block-triangular structure and $A_{11}$ corresponds to the Laplacian associated to a strongly connected graph, as desired. Furthermore, the block $P_{\widebar G_1}^\top R_{\widebar G_1} \in \mathbb R^{(N-n_1)\times n_1}$ may be split into two stacked sub-blocks. The upper one is of dimension $n_2\times n_1$ and contains the links that connect the  nodes in $G_1$ to  nodes in $G_2$; for the purpose of constructing \eqref{975}, we name this sub-block $A_{21}$. On the other hand, by the definition of $S_{\widebar G_1}$ and \eqref{203},
\begin{equation}
  \label{266}  P_{\widebar G_1}^\top S_{\widebar G_1} P_{\widebar G_1} = 
  \begin{bmatrix}
    \phantom{-}Q_{G_2} & 0 \\
    -R_{\widebar G_2} & S_{\widebar G_2}
  \end{bmatrix}
+ P_{\widebar G_1}^\top D_{\widebar G_1} P_{\widebar G_1}.
\end{equation}
The last term on the right-hand side of \eqref{266} is diagonal and may be split into two diagonal sub-blocks, {\it i.e.,} $P_{\widebar G_1}^\top D_{\widebar G_1} P_{\widebar G_1} =: $ blockdiag$\big[D'_{\widebar G_1} \ D''_{\widebar G_1}  \big]$. Then, we set $A_{22}$ in \eqref{975} to $A_{22} := Q_{G_2} + D'_{\widebar G_1}$ and we redefine $S_{\widebar G_2} := L_{\widebar G_2} + D_{\widebar G_2} + D_{\widebar G_1}''$. Thus, after \eqref{204} and \eqref{266}, and the previous definitions, we have
\begin{align}
  \begin{bmatrix}
     I_{n1} & 0  \\ 0 & P_{\widebar G_1}^\top
  \end{bmatrix}
  \begin{bmatrix}
    \phantom{-}A_{11} & 0 \\
    - R_{\widebar G_1} & S_{\widebar G_1}
  \end{bmatrix}
  \begin{bmatrix}
     I_{n1} & 0  \\ 0 & P_{\widebar G_1}
  \end{bmatrix}
\, = \
  \begin{bmatrix}
     \phantom{-}A_{11} & 0  & 0 \\ 
     -A_{21} & A_{22} & 0 \\
     \phantom{-}[\,*\,] & -R_{\widebar G_2} & S_{\widebar G_2}
  \end{bmatrix}.
\nonumber\\[-7pt] \ 
  \label{248}
\end{align}

In the matrix on the right-hand side of \eqref{248}, the entries of $R_{\widebar G_2}$ represent the edges connecting the nodes from $G_2\subset \widebar G_1$ to the rest of the sub-graph $\widebar G_1$, {\it i.e.,} $\widebar G_2$. Now, as previously remarked for $\widebar G_1$, $\widebar G_2$ may or may not contain a spanning tree. If it does, and $\widebar G_2$ is strongly connected, the matrix on  the right-hand side of \eqref{248} qualifies as the sought matrix in \eqref{975}. If $\widebar G_2$ contains a spanning tree, but is not strongly connected, Fact 1 applies to $\widebar G_2$ and generates a strongly connected graph $G_3\subset \widebar G_2$ and its complement $\widebar G_3 := \widebar G_2\backslash G_3 =  G \backslash \{G_1 \cup G_2 \cup G_3\}$. Then, we repeat the procedure above with the pertinent changes in the notation, {\it etc.} The process repeats as long as Fact \ref{fact1} applies, thereby generating a finite sequence of subgraphs $\{G_k\}$, with $k\in\{1,2,\ldots,m\}$ and  $m\leq N$, such that $G_k$ is the largest strongly connected sub-graph having incoming links {\it only} from subgraphs $G_{\ell}$ with $\ell\leq k-1$. For any such $k$, we obtain
\begin{equation}
  \label{272}  
  \begin{bmatrix}
     \phantom{-}A_{11} & \phantom{-}0  & \ \cdots & \cdots & 0 \\[-3pt] 
     -A_{21} & \phantom{-}A_{22} & \rotatebox{3}{$\ddots$} & & \vdots \\
     \vdots &  & \rotatebox{3}{$\ddots$} & \rotatebox{3}{$\ddots$}  &  \\[3pt]
     -A_{k-1,1} & \quad \cdots & & \!A_{kk} & 0 \\[3pt]
     \phantom{-}[\,*\,] & \cdots & \phantom{-}[\,*\,]  & -R_{\widebar G_k} & S_{\widebar G_k}
  \end{bmatrix}.
\end{equation}

 By construction, the lowest-rightest sub-block in the matrix on the right-hand side of \eqref{272} may be decomposed as $S_{\widebar G_k} = L_{\widebar G_k} + D_{\widebar G_k}$, where $D_{\widebar G_k}$ contains the weights of the links from the graphs $G_{\ell}$ with $\ell\leq k$ to $\widebar G_k$ and the previous arguments apply if $\widebar G_k$ contains a spanning tree. On the contrary, if $\widebar G_k$, for any $k\geq 1$, does not contain a spanning tree, the following applies. 

\noindent \underline{\it Case 2}: Assume that  $\widebar G_k$, with $k\geq 1$, does not contain a spanning tree. It follows that the associated Laplacian $L_{\widebar G_k}$ has $\mu_k > 1$ null eigenvalues. 
After \cite[Theorem 3.2]{caughman2006kernels}---{\it cf.} \cite[Proposition 3]{monaco2019multi}, it follows that there exists a permutation matrix $T$ such that 
\begin{equation}
  \label{261} 
T^\top L_{\widebar G_k} T = 
  \begin{bmatrix}
    L_{\widebar G_k}^{1} & 0 & \ \cdots & 0 \\[4pt]
    0 & L_{\widebar G_k}^{2} & & \\[-17pt]
           & & \rotatebox{10}{$\ddots$}& \vdots\\[-3pt]
    \vdots & & \rotatebox{10}{$\ddots$}& 0 \\[3pt]
    -M_{\mu_k + 1, 1} &  \cdots & -M_{\mu_k + 1, \mu_k} &\  M_{\mu_k + 1,\mu_k+1} 
  \end{bmatrix},
\end{equation}
where each block $L_{\widebar G_k}^{i}$, with $i \leq \mu_k$ corresponds to a Laplacian matrix associated to a sub-graph of $\widebar G_k$, that contains a spanning tree and that we denote $G_{k+i}$, for all $i \in \{1,2,\ldots,\mu_k\}$. Therefore, each  $L_{\widebar G_k}^{i}$ corresponding to a strongly connected graph $G_{k+i}$ may be placed in the appropriate order in the block diagonal of a block-triangular matrix of the form \eqref{975}, hence renamed $A_{jj}$. On the other hand, for each $G_{k+i}$ that is not strongly connected, Fact \ref{fact1} above applies, so we proceed as in Case 1. The sub-block $M_{\mu_k + 1,\mu_k+1}$ may be decomposed into $M_{\mu_k + 1,\mu_k+1} := L_{\mu_k + 1,\mu_k+1} + D_{\mu_k + 1,\mu_k+1}$, where $L_{\mu_k + 1,\mu_k+1}$ is a Laplacian and  $D_{\mu_k + 1,\mu_k+1}$  is a degree (diagonal semi-positive definite) matrix. $L_{\mu_k + 1,\mu_k+1}$ corresponds to a graph that  may or may not have a spanning tree, so either Case above applies.  

Since the graph $G$ has a finite number of nodes $N$, the processes described in Cases 1 and 2 above finish when either $\widebar G_k$ in Case 1 or the graph with Laplacian $L_{\mu_k + 1,\mu_k+1}$ in Case 2 is strongly connected, so we set either $A_{m,m}:= S_{\widebar G_k}$ or $A_{m,m}:= M_{\mu_k + 1,\mu_k+1}$.  This event will surely occur because after sufficiently many iterations, either of those graphs may contain only one leaf node, which constitutes a strongly connected (trivial) graph with  incoming edges. 



\end{proof}

\begin{proof}[of Lemma \ref{claim2}]
We follow the proof-lines of \cite[Proposition 2]{DYNCON-TAC16}. Under the assumption that $u_i \mapsto x_i$ defines a semi-passive map, for any $i\leq N$ there exists a radially unbounded storage function $V_i: \mathbb{R} \to \mathbb{R}_{+}$, a continuous function $H_i$, a positive continuous function $\psi_i$, and a positive constant $\rho_i$, such that the total derivative along the trajectories of \eqref{963} yields
\begin{equation}\label{eq890}
\dot V_i (x_i) \leq  u_i^\top x_i - H_i (x_i),
\end{equation}
where $H_i(x_i) \geq \psi_i(|x_i|)$ for all $|x_i| \geq \rho_i$. Next, let  $V_\Sigma (\bar x) :=  \sum_{i=1}^N \mu_i V_i(x_i)$, where $\mu_i$ corresponds to the $i$th element of the left eigen-vector associated to the zero eigen-value of $L$. Since by assumption the network is strongly connected, $\mu_i > 0$ for all $i \leq N$. Then, using \eqref{eq890} we see that 
\begin{equation}\label{eq891}
\dot V_\Sigma (\bar x) \leq  \sum_{i=1}^N \mu_i u_i^\top x_i -  \sum_{i=1}^N \mu_i H_i (x_i).
\end{equation}
The first term on the right-hand side of \eqref{eq891} gives
\begin{equation}\label{eq892}
 \sum_{i=1}^N \mu_i u_i^\top x_i  = u^\top\mathcal M \bar x,
\end{equation}
where $\mathcal M:=\mbox{diag}[\mu_i]$. Then, setting 
\[
u = -L\bar x - D\bar x + \sum_{j\leq p}B_j v_j,
\]
it follows that the derivative of $V_\Sigma(\bar x)$ along the trajectories of \eqref{121} satisfies
\begin{align}
\nonumber 
 \dot V_\Sigma (\bar x) \leq & -  \sum_{i=1}^N \mu_i H_i (x_i) -\bar x^\top L^\top \mathcal M \bar x\\
& \ - \bar x^\top D \mathcal M \bar x +  \Big[\sum_{j\leq p}B_j v_j \Big]^\top  \mathcal M \bar x.
  \label{209}  
\end{align}

Now, since the units are semi-passive, for each $i\leq N$, there exists $\rho_i > 0$ such that  $H_i(x_i) \geq \psi_i(|x_i|)$ for all $|x_i| \geq \rho_i$. Then, let $\bar \rho := \max_i \{ \rho_i \}$; it follows that 
\begin{equation}
  \label{216} -  \sum_{i=1}^N \mu_i H_i (x_i) \leq  - \sum_{i=1}^N\mu_i \psi_i (x_i),  
\end{equation}
for all all $|x_i| \geq  \bar \rho$. 

Furthermore, since the graph is strongly connected, $L^\top \mathcal M + \mathcal M L$ is positive semi-definite---{\it cf.} \cite[Proof of Proposition 2]{DYNCON-TAC16}. Hence, the second term on the right-hand side of \eqref{209} is non-positive. 

For the last two terms on the right-hand side of  \eqref{209} we observe that by the definition of $d_k$ and the fact that all the elements of any $B_j$ are non-negative, we have  $d_k = 0$ if and only if $[B_j]_{k\ell} = 0$ for all $\ell \leq M_j$ and for all $j\leq p$. Therefore, the third and fourth terms on the right-hand side of \eqref{209} satisfy 
\begin{equation}
  \label{225} 
- \bar x^\top D \mathcal M \bar x +  \Big[\sum_{j\leq p}B_j v_j \Big]^\top  \mathcal M \bar x
 \leq 
- \sum_{i=1}^N \big[ c_{1i} x_i^2  -  c_{2i} |x_i|  \big],
\end{equation}
where $c_{1i}$, $c_{2i} \geq 0$ and $c_{1i} = 0$ if and only if $c_{2i}= 0$. Therefore, for any $i\leq N$ there exists $\eta_i \geq 0$ such that for all $|x_i| \geq \eta_i$, $c_{1i} x_i^2 \geq  c_{2i} |x_i|$. Thus, 
\begin{equation}
\dot V_\Sigma (x)  \leq - \sum_{i=1}^N\mu_i 
\psi_i (|x_i|) \leq 0
\end{equation}
for all $|x_i| \geq \max\{\rho_i, \, \eta_i\}$. 
%
%
We conclude that if for any  $i\leq N$,   $|x_i(t)| \to \infty$ then there exists $T >0$ such that for all $t \geq T$, we have  $\dot V_\Sigma (x(t)) \leq 0$ for all $t\geq T$. The statement follows. 

\end{proof}

\bibliographystyle{IEEEtran}
\bibliography{../../../Biblio/Biblio_Anes/biblio_anes}

\begin{thebibliography}{1}
\providecommand{\url}[1]{#1}
\csname url@samestyle\endcsname
\providecommand{\newblock}{\relax}
\providecommand{\bibinfo}[2]{#2}
\providecommand{\BIBentrySTDinterwordspacing}{\spaceskip=0pt\relax}
\providecommand{\BIBentryALTinterwordstretchfactor}{4}
\providecommand{\BIBentryALTinterwordspacing}{\spaceskip=\fontdimen2\font plus
\BIBentryALTinterwordstretchfactor\fontdimen3\font minus
  \fontdimen4\font\relax}
\providecommand{\BIBforeignlanguage}[2]{{%
\expandafter\ifx\csname l@#1\endcsname\relax
\typeout{** WARNING: IEEEtran.bst: No hyphenation pattern has been}%
\typeout{** loaded for the language `#1'. Using the pattern for}%
\typeout{** the default language instead.}%
\else
\language=\csname l@#1\endcsname
\fi
#2}}
\providecommand{\BIBdecl}{\relax}
\BIBdecl

\bibitem{Porgmosky_semipassivity}
A.~Pogromsky, ``Synchronization and adaptive synchronization in semi-passive
  systems,'' in \emph{Proc. 1st Int. Conf. Control of Oscillations and Chaos},
  vol.~1, 1997, pp. 64--68 vol.1.

\bibitem{DYNCON-TAC16}
E.~Panteley and A.~Lor\'{\i}a, ``Synchronization and dynamic consensus of
  heterogeneous networked systems,'' \emph{IEEE Trans. on Automatic Control},
  vol.~62, no.~8, pp. 3758--3773, 2017.

\bibitem{caughman2006kernels}
J.~S. Caughman and J.~Veerman, ``Kernels of directed graph {L}aplacians,''
  \emph{The Electronic Journal of Combinatorics}, vol.~13, no.~1, p. R39, 2006.

\bibitem{monaco2019multi}
S.~Monaco and L.~R. Celsi, ``On multi-consensus and almost equitable graph
  partitions,'' \emph{Automatica}, vol. 103, pp. 53--61, 2019.

\end{thebibliography}

\end{document}